\documentclass[12pt, english, a4paper]{amsart}

\usepackage[T1]{fontenc}
\usepackage[utf8]{inputenc}
\usepackage[english]{babel}

\usepackage{accents}
\usepackage{caption}
\usepackage{comment}
\usepackage{enumerate}
\usepackage[hmargin=2.5cm,vmargin=3.2cm]{geometry}
\usepackage{graphics}
\usepackage[colorlinks]{hyperref}
\hypersetup{allcolors = black}
\usepackage{url}

\usepackage{amsmath}
\usepackage{amssymb}
\usepackage{amsthm}
\usepackage{fourier}
\usepackage{mathtools}
\usepackage{nicematrix}

\usepackage{commath}

\theoremstyle{plain}
\newtheorem{theorem}[]{Theorem}[section]

\newtheorem{lemma}[theorem]{Lemma}
\newtheorem{corollary}[theorem]{Corollary}
\newtheorem{conjecture}[theorem]{Conjecture}
\newtheorem{proposition}[theorem]{Proposition}

\numberwithin{equation}{section}

\newcommand{\R}{\mathbb{R}}
\newcommand{\Pro}{\mathbb{P}}
\newcommand{\dd}[1]{\,\mathrm{d}#1}
\newcommand{\mbf}[1]{\mathbf{#1}}
\newcommand{\one}[1]{\mathbf{1}_{#1}}
\newcommand{\zero}[1]{\mathbf{0}_{#1}}
\newcommand{\Proj}[2]{{#2}|_{#1}}
\newcommand{\vol}[2]{\mathrm{Vol}_{#1}\left(#2\right)}

\newcommand{\vv}{\mathbf{v}}

\newcommand{\eps}{\varepsilon}

\DeclareMathOperator{\sinc}{sinc}

\DeclarePairedDelimiterX\sca[2]{\langle}{\rangle}{#1,#2}
\makeatletter
\let\oldsca\sca
\def\sca{\@ifstar{\oldsca}{\oldsca*}}
\makeatother



\begin{document}
\title{Non-diagonal critical central sections of the cube}
\author{Gergely Ambrus \and Barnabás Gárgyán}
\thanks{
Research of G. A. was partially supported by ERC Advanced Grant "GeoScape" no.  882971,  by the Hungarian National Research grants no. NKFIH KKP-133819 and no. NKFIH K-147145,  and by project no. TKP2021-NVA-09, which has been implemented with the support provided by the
Ministry of Innovation and Technology of Hungary from the National
Research, Development and Innovation Fund, financed under the
TKP2021-NVA funding scheme.}
\date{\today}

\maketitle

\begin{abstract}
We study the $(n-1)$-dimensional volume of central hyperplane sections of the $n$-dimensional cube $Q_n$. Our main goal is two-fold: first, we provide an alternative, simpler argument for proving that the volume of the section perpendicular to the main diagonal of the cube is strictly locally maximal for every $n \geq 4$, which was shown before by L. Pournin~\cite{Pournin-local}. Then, we prove that  non-diagonal critical central sections of $Q_n$ exist in all dimensions at least $4$. The crux of both proofs is an estimate on the rate of decay  of the Laplace-Pólya integral $J_n(r) = \int_{-\infty}^\infty \sinc^n t \cdot \cos (rt) \dd t$ that is achieved by combinatorial means. This also yields improved bounds  for Eulerian numbers of the first kind. 
\end{abstract}

\section{Introduction} \label{sec_intro}
Let $Q_n = \big[-\frac 1 2 , \frac 1 2 \big]^n$ denote the centered $n$-dimensional unit cube.  This paper is devoted to the study of the $(n-1)$-dimensional volume of central hyperplane sections of $Q_n$ (that is, sections of the form $Q_n \cap \vv^\perp$). Accordingly, for a given non-zero vector $\vv \in \R^n$, we introduce the {\em central section function}
	\begin{equation}\label{eq:vn}
		\sigma(\mbf{v})=\vol{n-1}{Q_n\cap\mbf{v}^\perp}.
	\end{equation}
Note that the  quantity above is invariant under scalings of $\vv$ by a non-zero factor, and by embeddings of $\vv$ into $\R^m$ with $m \geq n$ and replacing $Q_n$ by $Q_m$.

The function $\sigma(\mbf{v})$ has been studied intensively in the last 50 years. According to a natural conjecture, which had been popularized by Good \cite{Ball-Good}, minimal central sections are parallel to a facet of $Q_n$. This was proved by Hadwiger~\cite{Hadwiger} in 1972, while  Hensley~\cite{Hensley} gave an alternative proof a few years later. He also provided an upper bound on the volume of central hyperplane sections. Completing the characterization of global extrema, Ball~\cite{Ball86} proved that the maximal central sections are orthogonal to the main diagonal of a 2-dimensional face of~$Q_n$.

Our main goal is to study critical points of the functional $\sigma(\mbf v)$ on the unit sphere $S^{n-1}$ -- these will be referred to as {\em critical directions}, and the corresponding sections as {\em critical sections}. The latter were recently characterized by Ivanov and Tsiutsiurupa~\cite{Ivanov} and, independently, the first named author~\cite{Ambrus}. 
{\em Locally extremal} sections are also defined via the analogous property of their unit normal on $S^{n-1}$.

A central section is called {\em $k$-diagonal} if its normal vector is parallel to the main diagonal of a $k$-dimensional face of $Q_n$, where $ 1 \leq k \leq n$. Let $\one{n}$ denote the $n$-dimensional vector $(1, \ldots, 1)$.
 The standard $k$-diagonal unit direction is given by 
\begin{equation}\label{eq:diagdirections}
\mbf{d}_{n,k}:=  
\frac {1}{\sqrt{k}} \cdot \big(\one{k}, \zero{n-k} \big)
\end{equation}
for $k = 1, \ldots, n$, where, naturally, $\zero{n-k}$ is the zero vector of $\R^{n-k}$. 
Up to permuting  coordinates and changing signs, all $k$-diagonal directions are of the above form. Such normal vectors and the corresponding central  sections of $Q_n$ will simply be called {\em diagonal}. 
For special values of $k$, we simplify the notation by writing
\[\mbf d_{n,1}=: \mbf e_1 \ \text{ and } \ \mbf d_{n,n} =: \mbf d_n.\]
In particular, $\mbf d_n = \frac 1 {\sqrt{n}} \one{n}$.

As a special case of a more general result, Pournin~\cite{Pournin-local} proved by local optimization techniques that all the diagonal sections of $Q_n$ are strictly locally extremal whenever $n \geq 4$. Our first result provides an alternative proof for this fact for  main diagonal sections.

\begin{theorem}\label{thm:localmax}
The main diagonal section $ Q_n \cap \one{n}^\perp$ has strictly locally maximal volume among central sections of $Q_n$ for each $n \geq 4$.
\end{theorem}

It has been widely believed that all critical central sections of $Q_n$ need to be diagonal. This was verified for $n=2,3$ but disproved for $n=4$ in \cite{Ambrus}. Note that appending 0's to a lower dimensional critical direction  yields critical directions for $Q_n$. Thus, it is only of interest to ask for the existence of non-diagonal critical sections whose normal vector does not have any 0 coordinates -- equivalently, sections which are not parallel to any of the coordinate axes. 

We prove the existence of such non-diagonal critical central sections in {\em all} dimensions exceeding 3.

\begin{theorem}\label{thm:nondiag}
		For all $n\geq4$ there exist non-diagonal critical central sections of $Q_n$ whose normal vector has only non-zero coordinates.  
	\end{theorem}

Furthermore, in Section~\ref{sec:saddle} we demonstrate that the critical directions constructed for the proof of Theorem~\ref{thm:nondiag} are saddle points of $\sigma(\mbf{v})$ on $S^{d-1}$. This leaves open the question of existence of non-diagonal critical sections which are {\em locally extremal} with respect to the central section function $\sigma(\mbf{v})$. Based on numerical evidence, we suspect that there are no such examples. 

\begin{conjecture}\label{conj:all-diag}
All locally extremal central sections of $Q_n$ are diagonal for each $n \geq 2$.
\end{conjecture}

The protagonist of the subsequent arguments is the {\em Laplace-Pólya integral}
	\begin{equation}\label{eq:Jnr}
		J_n(r):= \frac 1 \pi \int_{-\infty}^\infty\sinc^nt\cdot\cos(rt)\dd{t},
	\end{equation}
where $n$ is a positive integer, $r \in \R$, and $\sinc$ denotes the unnormalized sine cardinal function, that is
	\[\sinc x:=\begin{cases}\dfrac{\sin x}{x}&\text{if}\ x\neq0,\\[5pt]1&\text{if}\ x=0.\end{cases}\]
The integral \eqref{eq:Jnr} appears in a number of diverse problems. In particular, Laplace~\cite{Laplace} studied it in connection with  probability theory, while Pólya~\cite{Pólya} focused on  its importance in statistical mechanics.

For $n \geq 2$, or $n=1$ and $r \neq \pm1$, $J_n(r)$ may also be expressed as the density of the following Irwin-Hall distribution:
	\begin{equation}\label{eq:Jnr=fsum}
	    J_n(r)=2\cdot f_{\sum_{i=1}^nX_i}(r)
	\end{equation}
where $X_1, \ldots, X_n$ are independent random variables uniformly distributed on $[-1, 1]$, and $f_X(.)$ denotes the probability density function of a continuous random variable $X$. This stochastic interpretation, which is discussed in detail in Section~\ref{sec:prelim}, also  implies  that $J_n(r)=0$ for $|r| \geq n$ when $n  \geq 2$.

The crux of the subsequent arguments is a precise  estimate on the rate of decay of $J_n(r)$, which will be proved by entirely combinatorial means in Section~\ref{sec:integral}.

    \begin{theorem}\label{thm:ratio}
        Let $n\geq 4$ and $r$ be integers satisfying $-1\leq r\leq n-2$. Then
        \begin{equation}\label{eq:Jeq}
            \frac{J_n(r+2)}{J_n(r)} \leq c_{n,r},
        \end{equation}
        where
        \begin{equation}\label{eq:cnr}
            c_{n,r}=\dfrac{(n-r-2)(n-r)(n-r+2)}{(n+r)(n+r+2)(n+r+4)}.
        \end{equation}
    \end{theorem}

As an immediate corollary we derive the following bound.
   \begin{corollary}\label{cor:ineq}
		For each $n\geq2$,
		\begin{equation}\label{eq:ineq-Jn}
			(n+3)J_{n+2}(0)<(n+2)J_{n}(0).
		\end{equation}
	\end{corollary}
 
We note that for even values of $n$, Lesieur and Nicolas \cite{Lesieur-Euler} proved the slightly stronger estimate\footnote{We note that this bound and further esimates on the Laplace-Pólya integral can also be obtained by a combinatorial method similar to the present proofs, which is to be published in a subsequent paper.} 
\begin{equation}\label{eq:LNineq}
			(n+2)J_{n+2}(0)<(n+1)J_{n}(0)
\end{equation}
(also see~\eqref{eq:LN1}) by an intricate argument involving fine estimates for the power series expansion of Eulerian numbers of the first kind.

Further history of the problem and related results, in particular, the study of non-central and lower dimensional sections, are excellently surveyed in \cite{Nayar-extremal}.

\section{Preliminaries} \label{sec:prelim}

In this section we  review some of the necessary tools along the lines of the articles \cite{Ambrus,Ball86,koldobsky-fourier}. Pólya~\cite{Pólya} proved that the central
section function $\sigma(\vv)$ of a {\em unit} vector  $\mbf{v}=(v_1,\ldots,v_n)\in S^{n-1}$ may be evaluated by the classical integral formula 
	\begin{equation}\label{eq:PVn}
		\sigma(\mbf{v})=\frac{1}{\pi}\int_{-\infty}^\infty\prod_{i=1}^n\sinc{(v_it)}\dd{t}.
	\end{equation}
One can shortly derive an extension for arbitrary non-zero normal vectors, and non-central sections,  via the probabilistic interpretation as follows. 

Let now $\mbf{v}\in\R^n\setminus\{\zero{n}\}$ be an arbitrary non-zero vector, and define $S(\mbf{v},r)$ as the intersection of $Q_n$ and a hyperplane orthogonal to $\mbf{v}$ at distance $\frac{r}{\abs{\mbf{v}}}$ from the origin, that is
	\begin{equation}\label{eq:hyperplane}
	S({\mbf{v}},r):=\big \{\mbf{q}\in Q_n:\ \sca{\mbf{q}}{\mbf{v}}=r \big \}.
	\end{equation}
Introduce the {\em parallel section function}
	\begin{equation}\label{eq:pasecf}
		s({\mbf{v}},r):=\vol{n-1}{S(\mbf{v},r)};
	\end{equation}
then $\sigma(\vv) = s(\vv, 0)$.

	For a continuous random variable $X$, let $\varphi_X(.)$ denote its characteristic function. Let now $X_1,\ldots,X_n$ be independent random variables distributed uniformly on $[-1,1]$. 
 The joint distribution of $(X_1,\ldots,X_n)$ induces the normalized Lebesgue measure on $2 Q_n = [-1, 1]^n$. Accordingly, for arbitrary $\mbf{v}=(v_1,\ldots,v_n)\in\R^n$ and $r\in\R$
	\[\Pro\Bigg(\abs{\sum_{i=1}^nv_iX_i-r}\leq\varepsilon\Bigg)=\frac{1}{2^n}\vol{n}{\mbf{q}\in 2 Q_n:\ \abs{\sca{\mbf{q}}{\mbf{v}}-r}\leq\varepsilon} = \frac{\eps}{|\vv|} s\Big(\vv, \frac r 2 \Big) + o(\eps)
 \]
 provided that $S\big(\vv, \frac r 2 \big)$ is not a facet of $Q_n$. Dividing by $\eps$ and letting $\varepsilon\to0$ leads to
	\begin{equation}\label{eq:fsum}
		2 f_{\sum_{i=1}^nv_iX_i}(r)=\frac{1}{ \abs{\mbf{ v}}}s\Big(\vv, \frac r 2 \Big)
	\end{equation}
which holds whenever the left hand side exists. As is well known, the characteristic function of $\sum_{i=1}^n v_iX_i$ is
	\begin{equation}\label{eq:phisum}
		\varphi_{\sum_{i=1}^nv_iX_i}(t)=\prod_{i=1}^n\sinc{(v_it)}, 
	\end{equation}
	hence by taking the inverse Fourier transform one derives that
	\begin{equation}\label{eq:fsum-int}
		f_{\sum_{i=1}^nv_iX_i}(r)=\frac{1}{2\pi}\int_{-\infty}^\infty\prod_{i=1}^n\sinc{(v_it)}\cdot\cos{(rt)}\dd{t}.
	\end{equation}
	Therefore by \eqref{eq:fsum} and \eqref{eq:fsum-int} we obtain the following integral formula for $s(\mbf{v},r)$:
	\begin{equation}\label{eq:s-int}
		s\Big(\vv, \frac r 2 \Big)=\frac{\abs{\mbf{v}}}{\pi}\int_{-\infty}^\infty\prod_{i=1}^n\sinc{(v_it)}\cdot\cos{(rt)}\dd{t}.
	\end{equation}
	In particular,\footnote{Note that the factor $|\vv|$ is missing in \cite[formula (2.8)]{Ambrus}.}
	\begin{equation}\label{eq:sigma-int}
		\sigma(\mbf{v})=\frac{\abs{\mbf{v}}}{\pi}\int_{-\infty}^\infty\prod_{i=1}^n\sinc{(v_it)}\dd{t}
	\end{equation}
 which implies \eqref{eq:PVn}.

As a special case of \eqref{eq:s-int} one also obtains that the volume of sections orthogonal to the main diagonal can be expressed as 
	\begin{equation}\label{eq:Jnr=s}
	   s\Big(\one{n}, \frac r 2 \Big)=\sqrt{n} J_n(r).
	\end{equation}
In particular, central main diagonal sections are given by
\begin{equation}\label{eq:sigmaJn}
    \sigma(\one{n}) = \sqrt{n}J_n(0).
\end{equation}

Returning to \eqref{eq:PVn}, we derive that for unit vectors $\vv$
\begin{equation}\label{eq:pd}
  \frac{\partial}{\partial v_k} \sigma(\vv) = \frac 1 \pi \int_{-\infty}^\infty \prod_{ i \neq k} \sinc{(v_it)} \cdot \frac{\cos (v_k t) -\sinc{(v_kt)}}{v_k}\dd{t}
\end{equation}
(the differentiability property of the function $\sigma(\vv)$ is rigorously proven in the works of L. Pournin~\cite{Pournin-shallow,Pournin-local}).
Based on the Lagrange multiplier method, the following characterization was given in \cite{Ambrus} for critical points of $\sigma(\vv)$ on $S^{n-1}$ (note that in the present article we normalize $\sigma(\vv)$ differently):
    \begin{proposition}[\cite{Ambrus}, Formula (2.12)]\label{thm:Equ-2.12}
		The unit vector $\mbf{v}=(v_1,\ldots,v_n)\in S^{n-1}$ is a critical direction with respect to the central section function $\sigma(\mbf{v})$ if and only if up to permuting coordinates and changing signs\footnote{In \cite{Ambrus}, the trivial case $\mbf{v}=\mbf{e}_{1}$ was erroneously omitted.}  $\mbf{v}=\mbf{e}_{1}$, or
		\begin{equation}\label{eq:2.12}
			\sigma(\mbf{v})=\frac{1}{\pi (1-v_j^2)}\int_{-\infty}^\infty\prod_{i\neq j}\sinc{(v_it)}\cdot\cos{(v_jt)}\dd{t}
		\end{equation}
		holds for each $j=1,\ldots,n$.
	\end{proposition}
The argument also yields that at critical directions $\vv \in S^{n-1}$,
\begin{equation}\label{eq:partialder}
\frac{\partial}{\sigma v_i} \sigma(\mbf v)=-\sigma(\mbf v)\cdot v_i,
\end{equation}
see \cite[Proof of Theorem 1]{Ambrus}. Accordingly, the Lagrange function
\begin{equation}\label{eq:lagrangefunction}
\Lambda(\mbf v)=\sigma(\mbf v)+\tilde{\lambda}\cdot(\abs{\mbf v}^2-1)
\end{equation}
defined on $\R^n$ has a stationary point at $\vv$ with the Lagrange multiplier 
\begin{equation}\label{eq:lagrangempl}
   \Tilde{\lambda}=\frac{\sigma(\mbf v)}{2}. 
\end{equation}

We remark that, introducing the notation $\widetilde{\mbf{v}}_j=(v_1,\ldots,v_{j-1},v_{j+1},\ldots,v_n)\in\R^{n-1}$ and  using \eqref{eq:s-int}, equation \eqref{eq:2.12} translates to
	\begin{equation}\label{eq:sigma-stilde}
		\sigma(\mbf{v})=\frac{1}{(1-v_j^2)^{\frac32}}s\Big({\widetilde{\mbf{v}}}_j,\frac {v_j}{2}\Big).
	\end{equation}

\section{Properties of the Laplace-Pólya integral}\label{sec:integral}

In this section we study the integral formula $J_n(r)$ defined in \eqref{thm:nondiag} which is connected to various mathematical topics, see \cite{Bennett-sinc,Harumi-sinc,Medhurst-sinc,Pólya,Silberstein-sinc}. To us, its most prominent feature is the connection to diagonal sections provided by \eqref{eq:sigmaJn}.

	 The following explicit formula for $J_n(r)$ is well known, see \cite[pp. 165-170]{Laplace}:
	\begin{equation}\label{eq:Jnr-expli}
	    J_n(r)=\frac{1}{2^{n-1}(n-1)!}\sum_{i=0}^{\lfloor\frac{n+r}{2}\rfloor}(-1)^{i}\binom{n}{i}(n+r-2i)^{n-1}
	\end{equation}
which holds for  $|r|<n$.  Moreover, $J_n(r)$ can be expressed by the recursion
	\begin{equation}\label{eq:Jnr-recursion}
	    J_n(r)=\frac{n+r}{2(n-1)}J_{n-1}(r+1)+\frac{n-r}{2(n-1)}J_{n-1}(r-1)
	\end{equation}
 which was proved by Thompson \cite{Thompson}.
	Though this formula is valid for all $r\in\R$, we will only  use it for integer values of $r$. Since $J_n(r)$ is even in $r$, we specifically derive that
	\begin{equation}\label{eq:Jn0}
	    J_n(0)=\frac{n}{n-1}J_{n-1}(1)
	\end{equation}
	for $n\geq3$. Combined  with \eqref{eq:Jnr-recursion} this provides a simple way for computing the central values  $J_n(0)$ for small $n$'s, see Table~\ref{tab:J2-8(0)}.
	\begin{table}[h!] 
		\[\begin{NiceArray}{*{11}{c}}[hvlines-except-borders,cell-space-limits=5pt, columns-width =30pt]
			n&1&2&3&4&5&6&7&8&9&10\\
			J_{n}(0)&1&1&\frac{3}{4}&\frac{2}{3}&\frac{115}{192}&\frac{11}{20}&\frac{5887}{11520}&\frac{151}{315}&\frac{259723}{573440}&\frac{15619}{36288}
		\end{NiceArray}\]
		\caption{The value of $J_n(0)$ in the cases $1\leq n\leq10$.}
		\label{tab:J2-8(0)}
	\end{table}

The sequence $(J_n(0))_{n=1}^\infty$ possesses several monotonicity properties. Ball~\cite[Lemma 3]{Ball86} proved that $J_n(0)$ is monotone decreasing and converges to zero as $n\to\infty$. On the other hand, Aliev~\cite{Aliev} showed that  $nJ_n(0)$ is monotone increasing. According to  \eqref{eq:sigmaJn}, the value $\sqrt{n} J_n(0)$ is equal to the volume of the main diagonal section of $Q_n$.  Therefore, it is essential that
	\[\lim_{n\to\infty}\sqrt{n}J_n(0)=\sqrt{6\pi},\]
 see  \cite{Laplace,Pólya}.  
Based on Laplace's method for calculating the asymptotic expansion for $J_n(0)$, Bartha, Fodor and González Merino~\cite{Bartha-Fodor} proved that  the convergence is strictly monotone increasing for $n\geq3$, and noted that their method also implies that the sequence $\sqrt{n} J_n(0)$ is eventually concave. The asymptotic expansion up to order 3 reads as 
	\begin{equation}\label{eq:asymptotic}
			J_n(0)=\sqrt{\frac{6}{\pi n}}\bigg(1-\frac{3}{20n}-\frac{13}{1120n^2}+\frac{27}{3200n^3}+O\Big(\frac{1}{n^4}\Big)\bigg)
    \end{equation}
where the correct coefficients had been found (and subsequently corrected) in a series of papers \cite{Silberstein-sinc,Burgess-sinc,Grimsey-sinc,Goddard-sinc,Parker-sinc,Medhurst-sinc}). Even finer estimates were proved in \cite{Kerman,Schlage-sinc}.

A combinatorial interpretation of $J_n(r)$ stems from the connection with Eulerian numbers of the first kind $A(m,l)$ (which will simply be referred to as Eulerian numbers). These are recursively defined  \cite[pp. 240-243]{Comtet} for integers $m,l \geq 0$ by
\begin{align}\label{eq:Euler-recursion}
	\begin{split}
	    A(0,0)&:=1,\quad A(m,0):=0\ \text{for}\ m>0,\quad A(0,l):=0\ \text{for}\ l>0,\\[5pt] A(m,l)&=(m-l+1)A(m-1,l-1)+l A(m-1,l) \text{ for}\ m>0,\ l>0 .
	\end{split}
	\end{align}
Moreover they may be evaluated explicitly as
	\begin{equation}\label{eq:Euler-expli}
	    A(m,l)=\sum_{i=0}^l(-1)^{i}\binom{m+1}{i}(l-i)^m,
	\end{equation}
 see e.g. {\cite{carlitz-euler,Lesieur-Euler}}. Eulerian numbers can also be defined combinatorially \cite{Comtet} as the number of permutations of $\{1, \ldots, n\}$ in which exactly $l-1$ elements are greater than the previous element. This also shows the symmetry property
 \begin{equation}\label{eq:Euler-symm}
     A(m,l)=A(m,m-l+1).
 \end{equation}

 According to formulae \eqref{eq:Jnr-expli} and \eqref{eq:Euler-expli}, the following connection holds between the Laplace-Pólya integral and the Eulerian numbers:
	\begin{equation}\label{eq:Jnr-Euler}
	    J_n(r)=\frac{1}{(n-1)!}\,A\Big(n-1,\frac{n+r}{2}\Big)
	\end{equation}
	where $n\geq 2$ and $r $ is an integer s.t. $n+r$ is even. Due to this relation, Theorem~\ref{thm:ratio} leads to an estimate on the ratio between consecutive Eulerian numbers.

 	\begin{proposition}\label{thm:Lesieur-jobb}
	 For integers $l\geq 2$ and $m\geq 2l-1$
	    \begin{equation}\label{eq:Les-jobb}
	        A(m,l-1)\leq c_{m+1,m-2l+1}A(m,l),
	    \end{equation}
	    where $c_{n,r}$ is defined by formula {\eqref{eq:cnr}}.
	\end{proposition}

This strengthens the bound proved by Lesieur and Nicolas~\cite[Section 2.3, Theorem 3]{Lesieur-Euler} stating that for every $l\geq 2$ and $m\geq 2l-1$,
  \begin{equation}\label{eq:Lesieur}
	        A(m,l-1)<\Big(\frac{m-l}{m-l+2}\Big)^{m-2l+2}A(m,l).
	    \end{equation}
The authors also showed that~\cite[Section 3.5, Theorem 2]{Lesieur-Euler} 
\begin{equation}\label{eq:LN1}
\frac{m+1}{m+2}<\frac{m!}{(m+2)!}\cdot\frac{\max_{1\leq l\leq m+2}A(m+2,l)}{\max_{1\leq l\leq m}A(m,l)}<\frac{m+2}{m+3}
\end{equation}
for any odd $m$. Notice that
\[\frac{1}{m!}\max_{1\leq l\leq m}A(m,l)=\frac{1}{m!}A\Big(m,\Big\lfloor\frac m2\Big\rfloor+1\Big)=\begin{cases}J_{m+1}(1),&\text{if}\ m\ \text{is even}\\J_{m+1}(0),&\text{if}\ m\ \text{is odd}.\end{cases}\]
Therefore, by introducing $n=m+1$ and using \eqref{eq:Jn0}, inequality \eqref{eq:LN1} takes the  form
\begin{equation}\label{eq:LNestimate}
\frac{n(n-2)}{(n+2)^2}<\frac{J_n(2)}{J_n(0)}<\frac{n(n^2-2)}{(n+2)^3}
\end{equation}
which holds for any even $n$ greater than 3. Based on numerical calculations we do believe that these estimates are valid for every $n\geq3$. Note that the upper bound is slightly stronger than \eqref{eq:Jeq} for $r = 0$ and $n$ even --- yet, \eqref{eq:LNestimate} provides no estimate for  odd values of $n$, or $r \neq 0$, in contrast with Theorem~\ref{thm:ratio}.

\medskip 

Next, we establish the bound on the rate of decay of $J_n(r)$.

	    \begin{proof}[Proof of Theorem~\ref{thm:ratio}.]
        
        Suppose that $r$ is an integer satisfying $r\geq-1$. We will proceed by induction on $n$, with $n=4$ being the base case. 
        
        The values of $J_4(r)$ for $r=-1, \ldots, 4$ can be calculated by formula \eqref{eq:Jnr-expli}. Based on these, the $n=4$ case of \eqref{eq:Jeq} is easy to check, see Table~\ref{tab:ineq-J5k}.
        
        \begin{table}[h]
        \[\begin{NiceArray}{*{5}{c}}[hvlines-except-borders,cell-space-limits=5pt, columns-width =30 pt]
	    r&-1&\ 0 & 1 & 2 \\
            \frac{J_4(r+2)}{J_4(r)}&1&\frac14&\frac1{23}&0\\
            c_{4,r}&1&\frac14&\frac1{21}&0
	    \end{NiceArray}\]
	    \caption{Two sides of the inequality \eqref{eq:Jeq} for $n=4$ and $r=-1,\ldots,2$.}
	    \label{tab:ineq-J5k}
	    \end{table}
	
	    Suppose now that for some $n\geq4$,  \eqref{eq:Jeq} holds for each $-1\leq r\leq n-2$. We need to show      that 
             \begin{equation}\label{eq:proof-Jn+1-2}
            J_{n+1}(r+2)\leq c_{n+1,r}J_{n+1}(r)
        \end{equation}
        for each $r$ with $-1 \leq r \leq n-1$.

        For $r=-1$ or $r = n-1$, equality holds above: in the former case $c_{n+1, -1}=1$ and $J_{n+1}(-1) =J_{n+1}(1) $, while in the latter  $c_{n+1, n-1}=0$ and  $J_{n+1}(n+1) = 0$.
        
        Assume that $0 \leq r \leq n-3$. 
        By \eqref{eq:Jnr-recursion},
        \begin{align}\label{eq:J_{n+1}(r+3)}
            J_{n+1}(r+2)&=\frac{n+r+3}{2n}J_n(r+3)+\frac{n-r-1}{2n}J_n(r+1),\\
            J_{n+1}(r)&=\frac{n+r+1}{2n}J_n(r+1)+\frac{n-r+1}{2n}J_n(r-1). \notag
        \end{align}
        Applying the induction hypothesis for the pairs $(n,r+1)$ and  $(n,r-1)$ shows that
	     \begin{align*}
	        J_{n}(r+3)&\leq c_{n,r+1}J_n(r+1), \notag\\
	        \frac{1}{c_{n,r-1}}J_n(r+1)&\leq J_n(r-1). 
        \end{align*}
       Thus, \eqref{eq:proof-Jn+1-2} follows from the inequality
        \begin{equation*}\label{eq:Jc2}
            \frac{n+r+3}{2n}c_{n,r+1}+\frac{n-r-1}{2n}\leq c_{n+1,r}\Big(\frac{n+r+1}{2n}+\frac{n-r+1}{2n}\cdot\frac{1}{c_{n,r-1}}\Big).
        \end{equation*}
        Using \eqref{eq:cnr}, this simplifies to
        \begin{align*}\label{eq:Jc2-nn}
            \frac{(n-r-3)(n-r-1)(n-r+1)}{(n+r+5)(n+r+1)}&+(n-r-1)\leq\\[5pt]&\leq\frac{(n-r-1)(n-r+1)(n-r+3)}{(n+r+5)(n+r+3)}+\frac{(n-r+1)(n+r-1)}{(n+r+5)}.
        \end{align*}
     After combining the fractions, this takes the form
        \[0\leq16r^3+48r^2+32r\]
        which clearly holds.

Finally, when $r= n-2$, then due to $J_n(r+3)$ being 0, \eqref{eq:J_{n+1}(r+3)} reduces to
	    \[
      J_{n+1}(n)=\frac{1}{2n}J_n(n-1).
        \]
Therefore, using \eqref{eq:Jnr-recursion} for the term $J_{n+1}(n-2)$, \eqref{eq:proof-Jn+1-2} is seen to be equivalent to
\begin{equation}\label{jnrr}
\frac 1 {2n} J_n(n-1) \leq c_{n+1, n-2} \Big( \frac {2n -1}{2n}  J_n(n-1) + \frac 3 {2 n } J_n(n-3) \Big).
\end{equation}
The induction hypothesis for the pair $(n, n-3)$ implies that 
\[
\frac 1 {c_{n, n-3}} J_n(n-1) \leq J_n(n-3).
\]
Therefore, \eqref{jnrr} follows from the inequality
        \begin{equation*}
            1\leq c_{n+1,n-2}\Big({2n-1}+\frac{3}{c_{n,n-3}}\Big),
        \end{equation*}
       which, after substituting \eqref{eq:cnr}, simplifies to
\[
0 \leq 8 n^2  - 20 n + 3.
\]
As this holds for every $n \geq 3$, the proof is complete. \qedhere
        
    \end{proof}

    \begin{proof}[Proof of Corollary~\ref{cor:ineq}.]
	
	    The inequality can be easily confirmed in the cases of $n=2,\,3$ based on Table~\ref{tab:J2-8(0)}. Henceforth we suppose that $n\geq4$. Due to {\eqref{eq:Jn0}}, {\eqref{eq:Jnr-recursion}} and \eqref{eq:Jeq} we have 
\begin{equation}\label{eq:J(n+2)(0)}
\begin{split}
J_{n+2}(0)&=\frac{n+2}{n+1}J_{n+1}(1)=\\
&=\frac{(n+2)^2}{2n(n+1)}J_n(2)+\frac{n+2}{2(n+1)}J_n(0) \leq \\
&\leq \left( \frac{(n+2)^2}{2n(n+1)} \cdot c_{n,0} + \frac{n+2}{2(n+1)} \right)J_n(0)=\\
&= \frac{(n+2)(n^2 +2n -2)}{n (n+1)(n+4)}J_n(0).
\end{split}
\end{equation}
Therefore, \eqref{eq:ineq-Jn} is implied by the inequality
\[
\frac{(n+2)(n^2 +2n -2)}{n (n+1)(n+4)} < \frac{n+2}{n+3}
\]     
which holds for every $n \geq 1$. 
	\end{proof}
Finally, we prove the corresponding estimate for Eulerian numbers. 

    \begin{proof}[Proof of Proposition~\ref{thm:Lesieur-jobb}.]
	
	    Suppose that $r$ is an integer s.t. $r\geq-1$ and $n+r$ is even. According to equation {\eqref{eq:Jnr-Euler}} and Theorem~\ref{thm:ratio} we have
	    \[A\Big(n-1,\frac{n+r}{2}+1\Big)\leq c_{n,r}A\Big(n-1,\frac{n+r}{2}\Big).\]
            The symmetry property \eqref{eq:Euler-symm} 
            leads to the inequality
	    \[A\Big(n-1,\frac{n-r}{2}-1\Big)\leq c_{n,r}A\Big(n-1,\frac{n-r}{2}\Big).\]
	    Letting $m=n-1$ and $l=\dfrac{n-r}{2}$ implies {\eqref{eq:Les-jobb}}.
     \end{proof}

In order to demonstrate that \eqref{eq:Les-jobb} is indeed stronger than \eqref{eq:Lesieur} one has to prove that 
	    \begin{equation*}\label{eq:ineq-cL}
	        \frac{(l-1)l(l+1)}{(m-l+1)(m-l+2)(m-l+3)}<\Big(\frac{m-l}{m-l+2}\Big)^{m-2l+2}
	    \end{equation*}
     holds for each $l\geq2$ and $m\geq 2l-1$. By introducing $\mu=m-l+2$ (note that $\mu \geq l+1$), this transforms to  
	    \begin{equation}\label{eq:ineq-cL2}
	        \frac{(l-1)l(l+1)}{(\mu-1)\mu(\mu+1)}\cdot\Big(\frac{\mu-2}{\mu}\Big)^{l}<\Big(\frac{\mu-2}{\mu}\Big)^{\mu}.
	    \end{equation}
     Note that the right hand side is strictly monotone increasing in $\mu$. 

     We will prove that for any fixed $l\geq2$, the left hand side of \eqref{eq:ineq-cL2} is strictly monotone decreasing in $\mu$ for $\mu \geq l+2$. Indeed, its derivative in $\mu$ is given by
        \[-\frac{(l-1)l(l+1)}{(\mu-2)(\mu-1)^2\mu^2(\mu+1)^2}\cdot\Big(\frac{\mu-2}{\mu}\Big)^l\cdot(3\mu^3-2l\mu^2-6\mu^2-\mu+2l+2)\]
which is easily checked to be negative in the specified domain.

 Therefore, it suffices to verify \eqref{eq:ineq-cL2} for the cases $\mu=l+1,l+2$. These lead to the trivial inequalities  $\frac{l-1}{l+2}<\frac{l-1}{l+1}$ and $ \frac{l-1}{l+3}<\frac{l}{l+2}$, respectively.

\section{Main diagonal sections are strictly locally maximal}\label{sec:localmax}
By utilizing our combinatorial estimate on $J_n(r)$ stated in Theorem~\ref{thm:ratio}, we prove one of the core results of the paper stating that main diagonal sections of $Q_n$ are strictly locally maximal, except for the $3$-dimensional case.

\begin{proof}[Proof of Theorem~\ref{thm:localmax}]
To start with, note that Proposition~\ref{thm:Equ-2.12} and \eqref{eq:Jn0} imply that $\mbf d_{n}$ is a critical direction, hence we may apply the properties listed after  Proposition~\ref{thm:Equ-2.12}.

We will show that the function $\sigma(\vv)$ has a strict local maximum at $\vv = \mbf d_{n}$ subject to the constraint $|\vv| =1$. 
This is a constrained local optimization problem which can be solved by studying  the bordered Hessian matrix, i.e. the Hessian of the Lagrange function $\Lambda(\vv)$ defined by \eqref{eq:lagrangefunction}.

Since $\Lambda(\mbf v)=\sigma(\mbf v) + \frac {\sigma(\mbf v)}{2} \cdot (\abs{\mbf v}^2-1)$ because of \eqref{eq:lagrangefunction} and \eqref{eq:lagrangempl}, the bordered Hessian matrix is given by 
\begin{equation}\label{eq:borderedHessian}
H(\Lambda(\mbf v))=\begin{bNiceArray}{*{5}{c}}[margin,cell-space-limits=3pt,columns-width=1cm]
            0&2v_1&2v_2&\ldots&2v_n\\
            2v_1&\Block{4-4}{\dfrac{\partial^2\sigma}{\partial v_j \partial v_k}(\mbf v)+\sigma(\mbf v)\cdot\begin{cases}
                0,&\text{if}\ j\neq k\\
                1,&\text{if}\ j=k                
            \end{cases}}&&&\\
            2v_2&&&&\\
            \vdots&&&&\\
            2v_n&&&&
        \end{bNiceArray}
\end{equation}
Based on \eqref{eq:pd}, the entries of $H(\Lambda(\mbf v))$ apart from the first row and column may be calculated by
\begin{equation}\label{eq:betavdef}
   \beta_k(\vv) =
    \frac 1 \pi \int_{-\infty}^\infty\prod_{i\neq k}\sinc (v_i t)\cdot\bigg(\frac{2}{v_k^2}\Big(\sinc (v_kt)-\cos (v_kt)\Big)-\frac{(v_kt)^2}{v_k^2}\sinc (v_kt)+\sinc (v_kt)\bigg)\dd t
\end{equation}   
      along the diagonal $j=k$, and
\begin{equation}\label{eq:gammavdef}
    \gamma_{j,k}(\vv)   = \frac 1 \pi \int_{-\infty}^\infty\prod_{i\neq j,k}\sinc (v_it)\cdot\frac{\cos (v_jt)-\sinc (v_jt)}{v_j}\cdot\frac{\cos (v_kt)-\sinc (v_kt)}{v_k}\dd t
\end{equation}
        for the off-diagonal entries, i.e. $j\neq k$.

Consider now  the main diagonal direction $\mbf{d}_n = \frac 1{\sqrt{n}} \one{n}$, and let $H_m$ denote the principal minor of $H(\Lambda(\mbf{d}_n))$ of order $m$ for $m=3,\ldots,n$.  According to  \cite[Theorem 3.9.14.]{Moskowitz-funsev}), if $(-1)^{m-1}H_m>0$ for all $m=3,\ldots,n$, then $\sigma(.)$ has a strict local maximum on the constraint set $S^{n-1}$ at~$\mbf{d}_n$. 

 Substituting $\mbf v=\mbf{d}_n$ in \eqref{eq:betavdef} and \eqref{eq:gammavdef}, integrating by substitution for $v_k t$, applying the identity $\sin^2t=\frac{1-\cos(2t)}{2}$ and  employing \eqref{eq:Jnr-recursion} repeatedly  results in the formulae

  \begin{align*}
            \beta_k(\mbf{d}_n)
            &=
            \frac{\sqrt n}{\pi} \cdot \int_{-\infty}^\infty\Big(\sinc^{n-1}t\Big) \cdot\Big(2n(\sinc t-\cos t)-n t^2\sinc t+\sinc t\Big)\dd t=\\[5pt]&=
            {\sqrt n} \cdot \bigg((2n+1)J_n(0)-2n J_{n-1}(1)-\frac n2 J_{n-2}(0)+\frac n2 J_{n-2}(2)\bigg)=\\[5pt]&=
            {\sqrt n} \cdot \bigg(J_{n-2}(0)\cdot\Big(\frac{(2n+1)n}{2(n-1)}-n-\frac n2\Big) + J_{n-2}(2)\cdot\Big(\frac{(2n+1)n^2}{2(n-2)(n-1)}-\frac{n^2}{(n-2)}+\frac n2\Big)\bigg)=\\[5pt]&=
            \frac{n^{\frac32}}{2 (n-1)} \cdot \Big((4-n)J_{n-2}(0)+\frac{n^2+2}{n-2}J_{n-2}(2)\Big)
        \end{align*}
        
for any $k=1, \ldots, n$, and  
        \begin{align*}
           \gamma_{j,k}(\mbf{d}_n)
            &=
            \frac{n^{\frac32}}{\pi}\cdot\int_{-\infty}^\infty\Big(\sinc^{n-2}t\Big) \cdot \Big(\cos^2t-2\sinc t\cdot\cos t+\sinc^2t\Big)\dd t=\\[5pt]&=
             {n^{\frac32}}\cdot\Big(J_n(0)-2J_{n-1}(1)+\frac12J_{n-2}(0)+\frac12J_{n-2}(2)\Big)=\\[5pt]&=
            {n^{\frac32}}\cdot\bigg(J_{n-2}(2)\cdot\Big(\frac{n^2}{2(n-2)(n-1)}-\frac{n}{n-2}+\frac12\Big)+J_{n-2}(0)\cdot\Big(\frac{n}{2(n-1)}-1+\frac12\Big)\bigg)=\\[5pt]&=
             \frac{n^{\frac32}}{ 2(n-1)} \cdot  \big(J_{n-2}(0)-J_{n-2}(2)\big)
        \end{align*}
for $j \neq k$, $1 \leq j,k \leq n$.      
        Therefore, \eqref{eq:borderedHessian} yields that $ H(\Lambda(\mbf{d}_n))$ is of the form
        \begin{equation*}\label{eq:Hessian-a}
            H(\Lambda(\mbf{d}_n))=\begin{bNiceMatrix}
            0&\alpha&\alpha&\Cdots&\alpha\\
            \alpha&\beta&\gamma&\Cdots&\gamma\\
            \alpha&\gamma&\beta&\Cdots&\gamma\\
            \Vdots&\Vdots&\Vdots&\Ddots&\Vdots\\
            \alpha&\gamma&\gamma&\Cdots&\beta
            \end{bNiceMatrix}_{(n+1)\times(n+1)}
        \end{equation*}
        where
        \begin{equation}\label{eq:abg}
        \begin{gathered}
            \alpha=\frac{2}{\sqrt{n}},\quad \beta=\Big(  \frac{n^{\frac32}}{ 2(n-1)}  \Big)^{n-1}\cdot\Big((4-n)J_{n-2}(0)+\frac{n^2+2}{n-2}J_{n-2}(2)\Big),\\\text { and }\quad \gamma=\Big(  \frac{n^{\frac32}}{ 2(n-1)}  \Big)^{n-1}\cdot\Big(J_{n-2}(0)-J_{n-2}(2)\Big).
        \end{gathered}
        \end{equation}
        Thus,  all principal minors $H_m$ of $ H(\Lambda(\mbf{d}_n))$ have the same form, namely
        \begin{equation}\label{eq:H-det}
            H_m=\begin{vNiceMatrix}
            0&\alpha&\alpha&\Cdots&\alpha\\
            \alpha&\beta&\gamma&\Cdots&\gamma\\
            \alpha&\gamma&\beta&\Cdots&\gamma\\
            \Vdots&\Vdots&\Vdots&\Ddots&\Vdots\\
            \alpha&\gamma&\gamma&\Cdots&\beta
            \end{vNiceMatrix}_{m\times m}.
        \end{equation}
        Subtracting the second row from the ones below, expanding the resulting determinant along the first column, and finally adding the sum of all other columns of the remaining determinant to the first one results in an upper triangular matrix, hence 
 \[H_m=(-\alpha)\cdot\begin{vNiceMatrix}
            (m-1)\alpha&\alpha&\alpha&\Cdots&\alpha\\
            0&\delta&0&\Cdots&0\\
            0&0&\delta&\Cdots&0\\
            \Vdots&\Vdots&\Vdots&\Ddots&\Vdots\\
            0&0&0&\Cdots&\delta
        \end{vNiceMatrix}_{(m-1)\times (m-1)}\hspace{-1.25cm}=-(m-1)\alpha^2\delta^{m-2}\]
    with  $\delta=\beta-\gamma$. Therefore it suffices to show that $\delta<0$. For $n=4$ and $n = 5$, a direct calculation based on \eqref{eq:abg} yields that $\delta = -1$ and $\delta = - \frac 1 4$, respectively (we also remark that in the $n=3$ case one obtains $\delta = 0$). For $n \geq 6$, note that according to Theorem~\ref{thm:ratio},
    \[
    J_{n-2}(2) \leq \frac{n-4}{n+2} J_{n-2}(0),
    \]
       hence, by \eqref{eq:abg},
        \begin{align*}
        \delta&=\beta-\gamma=  \\
        &= \Big(  \frac{n^{\frac32}}{ 2(n-1)}  \Big)^{n-1}\cdot\Big((3 - n)J_{n-2}(0)+\frac{n(n+1)}{n-2}J_{n-2}(2)\Big)\leq
        \\
        &\leq  \Big(  \frac{n^{\frac32}}{ 2(n-1)}  \Big)^{n-1}\cdot\Big(- \frac{12}{n^2 - 4} J_{n-2}(0)\Big) < 0.
        \qedhere
        \end{align*}
\end{proof}

\section{Existence of non-diagonal critical sections}\label{sec:exist-non-diag}

This section is devoted to the proof of Theorem~\ref{thm:nondiag}, for which we present two approaches. The first one is only sketched below. 

\medskip

Consider the function $\sigma(\vv)$ on $S^{n-1}$. Theorem~\ref{thm:localmax} states that $\sigma(\vv)$ has a strict local maximum at $\vv = \mbf d_n$. On the other hand, the result of Ball~\cite{Ball86} shows that $\sigma(\vv)$ has a strict global maximum at $\vv=\mbf d_{n, 2}$ with $\sigma(\mbf d_{n, 2}) > \sigma(\mbf d_n)$. Let $\Gamma$ be the shorter great arc of $S^{n-1}$ connecting $\mbf d_n$ and $\mbf d_{n, 2}$ -- then $\Gamma$ consists of vectors of the form 
\begin{equation}\label{eq:vadef}
\vv_{n,2}(a) := (a,a,b, \ldots, b) \in S^{n-1}
\end{equation}
where $a \in \big [\frac 1 {\sqrt{n}}, \frac 1 {\sqrt{2}} \big]$. Let 
\begin{equation}\label{eq:sigma-hat}
\widehat{\sigma} (a) := \sigma \big(\vv_{n,2}(a)\big)
\end{equation}
be the restriction of $\sigma$ onto $\Gamma$. On the interval $\big [\frac 1 {\sqrt{n}}, \frac 1 {\sqrt{2}} \big]$, the function $\widehat{\sigma} (a)$ has a strict maximum at $\frac 1 {\sqrt{2}} $, while at $\frac 1 {\sqrt{n}}$ it has a strict local maximum. Because of \eqref{eq:pd},  $\widehat{\sigma} (a)$ is differentiable on the interval, hence it must have a local minimum at some $\xi_n \in \big (\frac 1 {\sqrt{n}}, \frac 1 {\sqrt{2}} \big)$. Using that $Q_n$ is symmetric with respect to the reflection over the line spanned by $\mbf d_{n, 2}$ in $\R^n$, one may show that $\vv(\xi_n)$ is a critical point of $\sigma$ on $S^{n-1}$.

\medskip

We decided to present the second proof below as it provides information about a wider class of normal vectors, it allows for the numerical calculation of the value of $\xi_n$, moreover, it is independent of the proof of maximality of $2$-diagonal sections  by Ball~\cite{Ball86}. The approach is based on the characterization result of \cite{Ambrus}, see Proposition~\ref{thm:Equ-2.12}.

\begin{proof}[Proof of Theorem~\ref{thm:nondiag}]
As a generalization of \eqref{eq:vadef}, we are going  to study a special class  of unit vectors which may be written in the form
\begin{equation*}\label{eq:vdef}
     \mbf{v}_{n,k}(a):=\begin{pNiceMatrix}a,&\Ldots&a,&b,&\Ldots&b\CodeAfter\UnderBrace[shorten,yshift=0pt]{1-1}{1-3}{k}\UnderBrace[shorten]{1-4}{1-6}{n-k}\end{pNiceMatrix}\in S^{n-1}
\end{equation*}
	 \vspace{5 pt}
 
  \noindent
	where $2\leq k\leq n-2$ and $a\in I_k:=\big[\frac{1}{\sqrt n},\frac{1}{\sqrt{k}}\big]$, furthermore $b$ is defined by
	\begin{equation}\label{eq:b_nka}
	    b:=b_{n,k}(a)=\sqrt{\dfrac{1-ka^2}{n-k}}.
	\end{equation}
  Here and later on we will always interpret $b$ as a function of $a$, $n$ and $k$, unless stated otherwise.
  
    Since $\mbf v_{n,k}(a)$ has at most two different coordinates, Proposition~\ref{thm:Equ-2.12}. implies that $\mbf v_{n,k}(a)$ is a critical direction if and only if
    $F_{n,k}(a)=0$, where
    \begin{align}\label{eq:F_nk}
	\begin{split}
		F_{n,k}(a):&=
		\frac{1}{\pi(1-a^2)}\int_{-\infty}^\infty\sinc^{n-k}{(bt)}\cdot\sinc^{k-1}{(at)}\cdot\cos{(at)}\dd{t}-\\[6pt]&-
		\frac{1}{\pi(1-b^2)}\int_{-\infty}^\infty\sinc^{n-k-1}{(bt)}\cdot\sinc^k{(at)}\cdot\cos (bt)\dd{t},
	\end{split}
    \end{align}
	defined on $I_k$. We will show the existence of $\xi_n \in \big (\frac 1 {\sqrt{n}}, \frac 1 {\sqrt{2}} \big)$ for which $F_{n,2}(\xi_n) = 0$. This suffices as $\vv_{n,2}(\xi_n)$ cannot be diagonal, neither can it have any 0 coordinates.
 
 The argument is divided to the following three lemmata whose proof is postponed to the end of the section. 
	
	\begin{lemma}\label{lem:1}
	    For each $2 \leq k \leq n-2$, $\dfrac{1}{\sqrt k}$ and $ \dfrac{1}{\sqrt n}$ are both zeros of $F_{n,k}$.
	\end{lemma}
In the next lemma we will consider the right-hand derivative of $F_{n,k}(.)$ at $\frac{1}{\sqrt{n}}$. In order to ease notation, this will be denoted by $F_{n,k}'\Big(\frac{1}{\sqrt{n}}\Big)$. We note that there is no difficulty in extending the domain of $F_{n,k}(.)$ hence this simplification is well justified.

  \begin{lemma}\label{lem:2}
	    For each $4\leq k\leq n-2$, $F_{n,k}$ is differentiable on the interval $I_k$. In the case of $k=2,3$, $F_{n,k}$ is differentiable on every compact subinterval of $I_k\setminus\big\{\frac{1}{\sqrt k}\big\}$. Moreover in both cases  \[F_{n,k}'\Big(\frac{1}{\sqrt{n}}\Big)<0.\]
	\end{lemma}
	
	\begin{lemma}\label{lem:3}
		For each $n\geq4$,  $F_{n,2}(a) \geq 0$ for every $a\in \big[ {\gamma_n}, \frac{1}{\sqrt2} \big] $ where $\gamma_n=\sqrt{\frac{n-2}{2n-3}}$. 
	\end{lemma}
    
   Note that the first  two statements hold for arbitrary $k$, while the third is proven  only for $k=2$. This is not merely a technical issue; based on numerical evidence we conjecture that there are no non-diagonal critical directions of the form $\vv_{n,k}(a)$ with $3 \leq k \leq n-2$, see Figure~\ref{fig:F6}.
	
	\begin{figure}[h!]
		\centering
		\begin{minipage}{.49\textwidth}
			\centering
			\includegraphics[width=.9\textwidth]{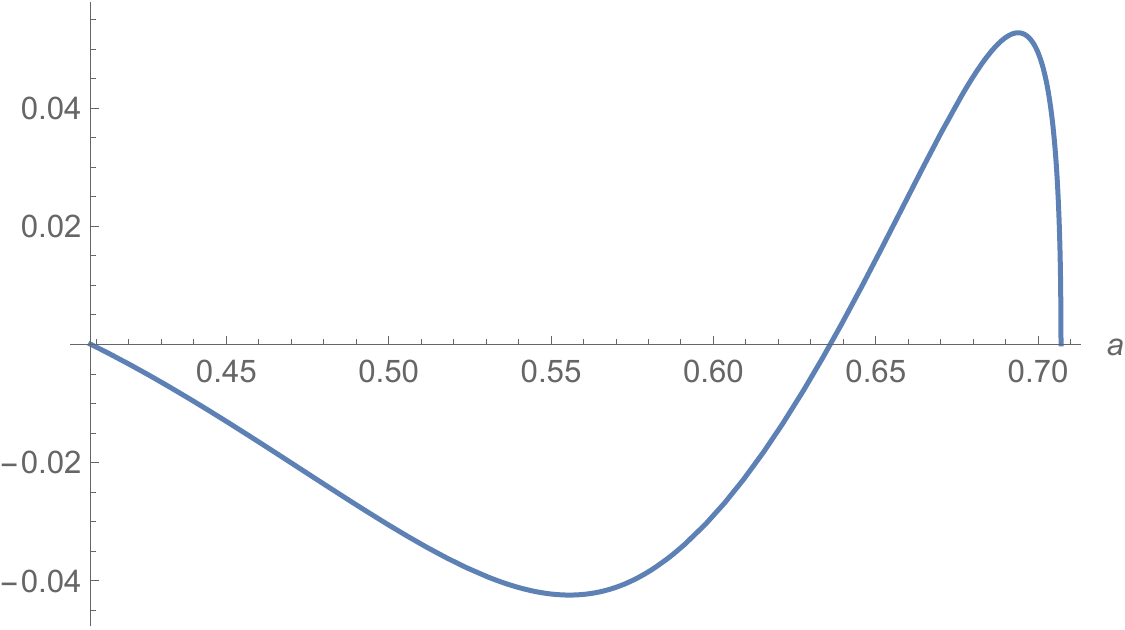}
		\end{minipage}
		\hfill
		\begin{minipage}{.49\textwidth}
			\centering
			\includegraphics[width=.9\textwidth]{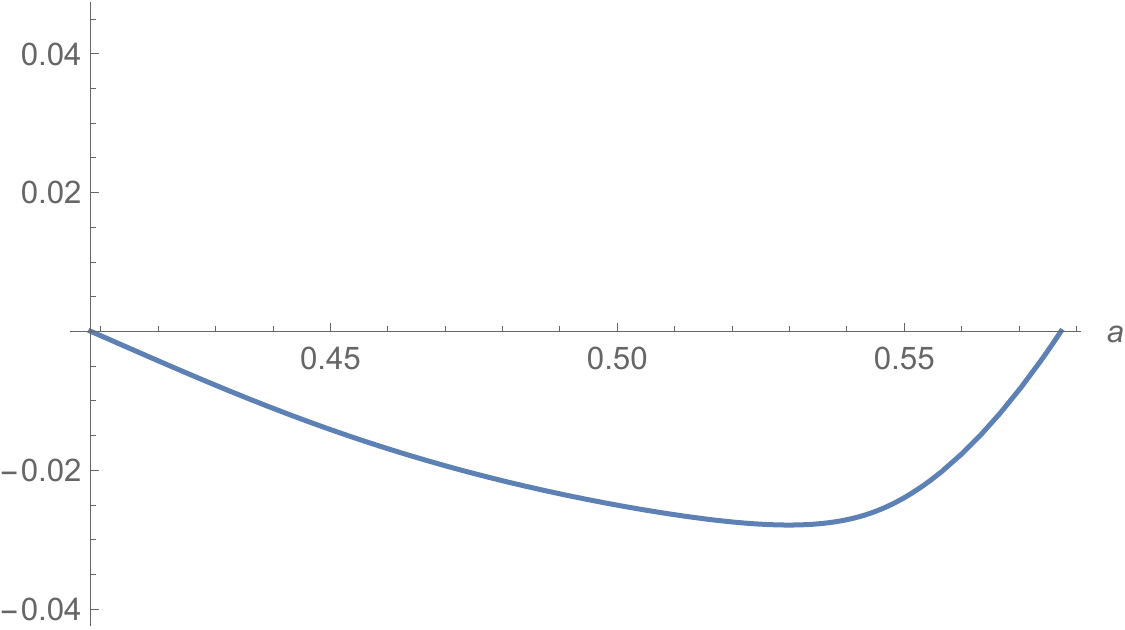}
		\end{minipage}
        \caption{Difference between the behavior of $F_{6,2}(a)$ and $F_{6,3}(a)$ }
        \label{fig:F6}
	\end{figure}

Once the above lemmata are established, the proof is immediate: Take $\gamma_n$ provided by Lemma~\ref{lem:3}. Then $F_{n, 2} ({\gamma_n}  ) \geq 0$, and $\gamma_n >\frac{1}{\sqrt{n}} $ as  $n \geq 4$. Moreover, acccording to Lemma~\ref{lem:2},  $F_{n, 2}$ is differentiable on the interval $\big[\frac{1}{\sqrt{n}}, {\gamma_n} \big]$ and for its right-hand derivative, $F_{n,2}'\big(\frac{1}{\sqrt{n}}\big)<0$ holds. Lemma~\ref{lem:1} yields that $F_{n,2}\big(\frac{1}{\sqrt{n}}\big)=0$, thus there exists some  $\eps>0$ with $F_{n,2}\big(\frac{1}{\sqrt{n}}+\eps\big) <0$ and $\frac{1}{\sqrt{n}}+\eps < \gamma_n$.
Applying the mean value theorem~\cite{Bolzano} on the interval $\big[ \frac{1}{\sqrt{n}}+\eps, {\gamma_n}  \big]$ yields the existence of $\xi_n$ with the prescribed property $F_{n,2}(\xi_n) = 0$. 
\end{proof}

Our proof only assures the existence of a nontrivial zero $\xi_n$. Based on numerical calculations  we suspect that there is only one suitable zero in the prescribed interval for each $n \geq 4$. Table~\ref{tab:n-xi} summarizes the numerical value of $\xi_n$ for small dimensions. For $n=4$, $\xi_4 = \sqrt{\frac25}$, which yields a non-diagonal critical direction parallel to $(1,1,2,2)$ \cite{Ambrus}; for larger dimensions, these values do not seem to follow such a nice pattern anymore.  
	
	\begin{table}[h!]
	\[\begin{NiceArray}{*{8}{c}}[hvlines-except-borders, cell-space-limits=5pt, columns-width = 44pt]
	    n&4&5&6&7&8&9&10\\
	    \xi_n&0.632455&0.634265&0.636071&0.636935&0.637520&0.637921&0.638219
	\end{NiceArray}	\]
		\caption{Numerical value of the zero $\xi_n$ in the dimensions $4\leq n\leq10$.}
        \label{tab:n-xi}
	\end{table}

\begin{proof}[Proof of Lemma~\ref{lem:1}]
The statement follows from the fact that all diagonal sections are critical, which is implied by Proposition~\ref{thm:Equ-2.12}, formula \eqref{eq:Jn0}, and the remark preceding Theorem~\ref{thm:nondiag}.
\end{proof}

\begin{proof}[Proof of Lemma~\ref{lem:2}]
For the integrand of $F_{n,k}(a)$ we introduce the notation
	\begin{equation}\label{eq:f_nk}
		f_{n,k}(a,t):=\frac{1}{1-a^2}\cdot\sinc^{n-k}{(bt)}\cdot\sinc^{k-1}{(at)}\cdot\cos{(at)}-\frac{1}{1-b^2}\cdot\sinc^{n-k-1}{(bt)}\cdot\sinc^k{(at)}\cdot\cos{(bt)}
	\end{equation}
	which is defined on $I_k\times\R$.   Due to the properties of the $\sinc$ function, $f_{n,k}(a,t)$ is continuous in both of its variables. Furthermore the function
	    \begin{equation}\label{eq:phi_nk}
			\psi_{n,k}(t):=\begin{cases}2&\text{if}\ \abs{t}\leq1,\\[5pt]\frac{2}{t^{n-1}}&\text{if}\ \abs{t}\geq1\end{cases}
		\end{equation}
		is an integrable majorant of $f_{n,k}(a,t)$ for all $a\in I_k$.
  We need to show that $\frac{\partial f_{n,k}(a,t)}{\partial a}$ is also continuous and it has an integrable majorant. In order to calculate the derivative of $f_{n,k}(a,t)$ with respect to $a$ on $I_k$ we will apply the  differentiation rules
	    \[(\sinc x)'=\frac{1}{x}(\cos x-\sinc x)\quad \text{and}\quad  b'=-\frac{kab}{1-ka^2}.\]
	    These yield that
    	\begin{align*}
    		\frac{\partial f_{n,k}(a,t)}{\partial a}&=
    		\frac{2a}{(1-a^2)^2}\cdot\sinc^{n-k}{(bt)}\cdot\sinc^{k-1}{(at)}\cdot\cos{(at)}+\\[5pt]&+
    		\frac{ka}{1-a^2}\cdot\sinc^{n-k-1}{(bt)}\cdot\frac{\sinc{(bt)}-\cos{(bt)}}{b^2}\cdot\sinc^{k-1}{(at)}\cdot\cos{(at)}+\\[5pt]&+
    		\frac{k-1}{1-a^2}\cdot\sinc^{n-k}{(bt)}\cdot\sinc^{k-2}{(at)}\cdot\frac{\cos{(at)}-\sinc{(at)}}{a}\cdot\cos{(at)}-\\[5pt]&-
    		\frac{1}{a(1-a^2)}\cdot\sinc^{n-k}{(bt)}\cdot\sinc^{k-2}{(at)}\cdot\sin^2{(at)}+\\[5pt]&+
    		\frac{2ka}{(n-k)(1-b^2)^2}\cdot\sinc^{n-k-1}{(bt)}\cdot\sinc^k{(at)}\cdot\cos{(bt)}-\\[5pt]&-
    		\frac{(n-k-1)ka}{(n-k)(1-b^2)}\cdot\sinc^{n-k-2}{(bt)}\cdot\frac{\sinc{(bt)}-\cos (bt)}{b^2}\cdot\sinc^k{(at)}\cdot\cos{(bt)}-\\[5pt]&-
    		\frac{k}{1-b^2}\cdot\sinc^{n-k-1}{(bt)}\cdot\sinc^{k-1}{(at)}\cdot\frac{\sinc{(at)}-\cos{(at)}}{a}\cdot\cos{(bt)}-\\[5pt]&-
    		\frac{k}{(n-k)(1-b^2)a}\cdot\sinc^{n-k-2}{(bt)}\cdot\sinc^{k-2}{(at)}\cdot\sin^2{(at)}.
    	\end{align*}
	    Clearly $\frac{\partial f_{n,k}(a,t)}{\partial a}$ is continuous on $I_k\times\R$ since	\[\lim_{a\to{1}/{\sqrt{k}}}\frac{\sinc{(bt)}-\cos{(bt)}}{b^2}=\frac{t^2}{3}.\]
	    When $4 \leq k \leq n-2$, for each term above an integrable majorant may be given similarly to \eqref{eq:phi_nk}, which yields that $\frac{\partial f_{n,k}(a,t)}{\partial a}$ has an integrable majorant on the whole $I_k$. However, for $k=2,3$, this property only holds on any compact subinterval of $I_k\setminus\big\{\frac1{\sqrt k}\big\} = \big [ \frac 1 {\sqrt{n}}, \frac 1 { \sqrt{k}} \big) $, since in these cases the term 
        \begin{gather*}
            \lim_{a\to1/\sqrt k}\frac{ka}{1-a^2}\cdot\sinc^{n-k-1}(bt)\cdot\frac{\sinc (bt)-\cos (bt)}{b^2}\cdot\sinc^{k-1}(at)\cdot\cos (at)=\\=
            \frac{k^{\frac32}}{3(k-1)}t^2\cdot\sinc^{k-1}\frac{t}{\sqrt k}\cdot\cos\frac{t}{\sqrt k}
        \end{gather*}
        is not integrable, neither is eliminated by the other summands.

      The Leibniz integral rule (see \cite[Theorem 3.2.]{Lang}) implies that $F_{n,k}(.)$ is differentiable on $I_k$ if $k\geq4$, and on every compact subinterval of $I_k\setminus\big\{\frac1{\sqrt k}\big\}$ if $k=2,3$. The derivative is given by
	    \begin{equation}\label{eq:Fn'}
	        F_{n,k}'(a)=\frac1\pi\cdot\frac{\partial{}}{\partial{a}}\int_{-\infty}^\infty f_{n,k}(a,t)\dd{t}=\frac1\pi\int_{-\infty}^\infty\frac{\partial f_{n,k}(a,t)}{\partial a}\dd{t}.
	    \end{equation}
	    Moreover $F_{n,k}'(.)$ is continuous from the right at $\frac{1}{\sqrt{n}}$ for every $k$ (see \cite[Theorem 3.1.]{Lang}). Accordingly,  the value of the right-hand derivative $F_{n,k}'\Big(\frac{1}{\sqrt{n}}\Big)$
	    may be evaluated by the substitution $a=\frac{1}{\sqrt{n}}$. Note that this leads to $b=a$, hence products of the $\sinc$ terms may be combined together. Thus each term of the integral $F_{n,k}'\big(\frac{1}{\sqrt{n}}\big)$ is of the form
        \[\frac1\pi\int_{-\infty}^\infty \alpha\sinc^m\frac{t}{\sqrt{n}}\cdot\cos{\frac{\beta}{\sqrt{n}}}\dd{t}=\alpha\sqrt{n}J_m(\beta)\]
        with some $\alpha,\beta\in\R$ and some positive integer $m$ (recall that $J_m(\beta)$ is defined by \eqref{eq:Jnr}). Therefore, from \eqref{eq:Fn'} we obtain that
        \begin{align*}
            F_{n,k}'\Big(\frac{1}{\sqrt{n}}\Big)&=
            \frac{(k-2)n^2}{2(n-1)}J_{n-2}(0)+
            \frac{k(n-k-2)n^2}{2(n-k)(n-1)}J_{n-2}(0)-
            \frac{kn^2}{n-1}J_{n-2}(0)+\\[5pt]&+
            \frac{n^2(n+1)}{(n-1)^2}J_{n-1}(1)-
            \frac{kn^2(n+1)}{(n-1)^2(n-k)}J_{n-1}(1)-\\[5pt]&-
            \frac{kn^2}{n-1}J_{n-2}(2)-
            \frac{kn^2}{2(n-1)}J_{n-2}(2)+
            \frac{kn^2}{2(n-1)}J_{n-2}(2).
        \end{align*}
	    This further simplifies to
    	\begin{align*}
    	    F_{n,k}'\Big(\frac{1}{\sqrt n}\Big)&=\frac{n^3}{(n-k)(n-1)}\Big(\frac{n+1}{n-1}J_{n-1}(1)-J_{n-2}(0)\Big) = \\ 
         &= \frac{n^3}{(n-k)(n-1)}\Big(\frac{n+1}{n}J_{n}(0)-J_{n-2}(0)\Big),
    	\end{align*}
     where \eqref{eq:Jn0} is used in the second step.     	Due to the estimate of Corollary~\ref{cor:ineq},  this is indeed negative for every $n\geq4$ and $2\leq k\leq n-2$.
\end{proof}

Next, we will apply a geometric argument in order to show the positivity of $F_{n,2}$ in a left neighborhood of $\frac 1 {\sqrt{2}}$. First we recall a standard useful volume formula. Let $H_{\mbf{u}}$ and $H_{\mbf{v}}$ be  hyperplanes with normal unit vectors $\mbf{u}$ and $\mbf{v}$, and denote by $\Proj{H_{\mbf{u}}}{.}$ the orthogonal projection onto $H_\mbf{u}$. Then 
	\begin{equation}\label{eq:proj}
	    \vol{n-1}{\Proj{H_{\mbf{u}}}{A}}=\abs{\sca{\mbf{u}}{\mbf{v}}}\cdot\vol{n-1}{A}
	\end{equation}
for any measurable set $A \subset H_{\mbf{v}}$.

\begin{proof}[Proof of Lemma~\ref{lem:3}] 
Throughout the proof, we set $k =2$. Accordingly,  \eqref{eq:b_nka} simplifies to 
\begin{equation}\label{eq:b_n2a}
    b = b_{n,2}(a) = \sqrt{\frac{1 - 2 a^2}{n-2}}\,.
\end{equation}
The constraint $a \geq \sqrt{\frac{n-2}{2n-3}}$ thus yields that
\begin{equation}\label{eq:abcons}
    \frac a b \geq n-2.
\end{equation}

According to \eqref{eq:F_nk} and \eqref{eq:s-int}
		\begin{equation}\label{eq:F=s-s}
			 F_{n,2}(a)=\frac{1}{(1-a^2)^{\frac32}}s\Big({\mbf{u}_1(a)},\frac a2\Big)-\frac{1}{(1-b^2)^{\frac32}}s\Big({\mbf{u}_2(a)},\frac b2\Big),
		\end{equation}
            where $\mbf u_1(a)=(a,b,\ldots,b)$ and $\mbf{u}_2(a)=(a,a,b,\ldots,b)$ are $(n-1)$-dimensional non-unit vectors obtained from $\mbf v_{n,2}(a)$ by deleting the first and third coordinate, respectively. Denote the above parallel section functions by
            \begin{equation}\label{eq:small-s1s2}
            s_1(a)=s\Big({\mbf{u}_1(a)},\frac a2\Big),\quad s_2(a)=s\Big({\mbf{u}_2(a)},\frac b2\Big)
            \end{equation}
            which, by \eqref{eq:hyperplane}, express the $(n-2)$-dimensional volume of the intersection of~$Q_{n-1}$ with the cross-sections
		\begin{equation}\label{eq:S1S2}
		S_1(a) =S\Big({\mbf{u}_1(a)},\frac a2\Big),\quad S_2(a) = S\Big({\mbf{u}_2(a)},\frac b2\Big).
		\end{equation}
		We may determine the functions $s_1(a)$, $s_2(a)$ exactly in an appropriate left neighbourhood of $a=\frac{1}{\sqrt2}$ by using suitably chosen orthogonal projections.

		We first study $s_1(a)$, the volume of the section $S_1(a)$. Consider the facet $F$ of $Q_{n-1}$ defined by
	     \[F=\Big\{\Big(\frac12,\mbf{y}\Big)\in Q_{n-1}:\ \mbf{y} \in Q_{n-2}\Big\}.
      \]
	   We will determine $\Proj{F}{S_1(a)}$, see Figure~\ref{fig:S1a}.
    \begin{figure}
        \centering
        \includegraphics[width=0.33\textwidth]{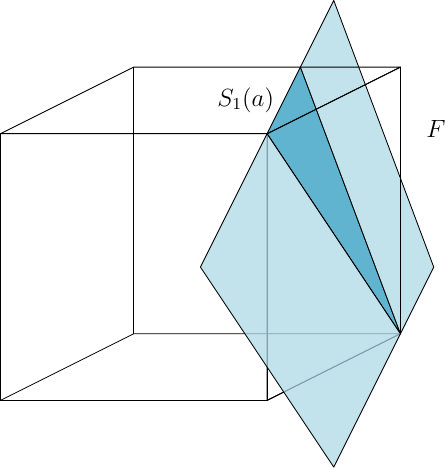}
        \caption{Section $S_1(a)$ and facet $F$ of $Q_{n-1}$.}
        \label{fig:S1a}
    \end{figure}
  
  Represent the points of $\R^{n-1}$ in the form of $(x, \mbf{y})$, where $x\in\R$ and $\mbf{y}\in\R^{n-2}$. Then, according to \eqref{eq:hyperplane} and \eqref{eq:S1S2}, the equation of the hyperplane in $\R^{n-1}$ corresponding to $S_1(a)$ can be written as
        \begin{equation*}\label{eq:S1-x}
            x=\frac12-\frac{b}{a}\sca{\mbf{y}}{\one{n-2}}.
        \end{equation*}
Note that on this hyperplane, $S_1(a)$ is described by the criteria $x \in [-\frac 12, \frac 12]$ and $\mbf{y} \in Q_{n-2}$. The first condition, by the equation above, transforms to  $0 \leq \sca{\mbf{y}}{\one{n-2}} \leq \frac{a}{b}$. Since  by \eqref{eq:abcons}, the upper limit is at least $n-2$, while 
$\max_{\mbf{y} \in Q_{n-2}}  \sca{\mbf{y}}{\one{n-2}} = \frac{n-2}{2}$,  the orthogonal projection of $S_1(a)$ to $F$ is specified by
	    \[\Proj{F}{S_1(a)}=\Big\{\Big(\frac12,\mbf{y}\Big)\in F:\ \sca{\mbf{y}}{\one{n-2}}\geq0\Big\}.\]
	    Note that this constitutes half of the facet $F$, hence
	    \begin{equation}\label{eq:Proj-F1}
	        \vol{n-2}{\Proj{F}{S_1(a)}}=\frac12\cdot\vol{n-2}{F}=\frac12.
	   \end{equation}
	    On the other hand, according to \eqref{eq:proj},
	    \begin{equation}\label{eq:vol-Proj-F1}
	        \vol{n-2}{\Proj{F}{S_1(a)}}=\abs{\sca{\frac{\mbf{u}_1(a)}{\abs{\mbf{u}_1(a)}}}{\mbf{e}_1}}\cdot\vol{n-2}{S_1(a)}=\frac{a}{\sqrt{1-a^2}}\cdot\vol{n-2}{S_1(a)},
	    \end{equation}
	    thus from equations {\eqref{eq:small-s1s2}}, {\eqref{eq:Proj-F1}} and {\eqref{eq:vol-Proj-F1}} we obtain that
	    \begin{equation}\label{eq:su1a}
	        s_1(a)=\frac{\sqrt{1-a^2}}{2a}.
	    \end{equation}

     Next, consider the quantity $s_2(a)$. Now represent the points of $\R^{n-1}$ in the form of $(\mbf{x}, \mbf{y})$, where $\mbf{x}\in\R^2$ and $\mbf{y}\in\R^{n-3}$. We will determine the orthogonal projection of  $S_2(a)$ onto the central diagonal section of $Q_{n-1}$ given as
	    \[D=\big\{(\mbf{x}, \mbf{y}) \in Q_{n-1}:\ \sca{\mbf{x}}{\one{2}}=0, \ \mbf{y} \in Q_{n-3}\big\}.\]
     
     The equation of the hyperplane corresponding to $S_2(a)$ is, by \eqref{eq:hyperplane} and \eqref{eq:S1S2}, 
	    \begin{equation}\label{eq:S2-x}
	        \sca{\mbf{x}}{\one{2}}=\frac{b}{a}\Big(\frac12-\sca{\mbf{y}}{\one{n-3}}\Big),
        \end{equation}
and $S_2(a)$ is described by $\mbf{x} \in Q_2$, $\mbf{y} \in Q_{n-3}$ in addition to the above equation.

Consider $\Proj{D}{S_2(a)}$.  We can calculate its volume  as follows. Let $L_1$ be the $2$-dimensional linear subspace of $\R^{n-1}$ which consists of vectors of the form $(\mbf{x}, \zero{n-3})$ with $\mbf{x} \in \R^2$. 		
Furthermore let $L_2=L_1^\perp$ in $\R^{n-1}$, an $(n-3)$-dimensional linear subspace. Then
	    \begin{equation}\label{eq:vol-Proj-F2-int}
	        \vol{n-2}{\Proj{D}{S_2(a)}}=\int_{L_2\cap Q_{n-1}}\vol{1}{S_2(a)\cap(\mbf{y}+L_1)}\dd{\mbf{y}},
	    \end{equation}
	    since during the projection of $S_2(a)$, the length of $S_2(a)\cap(\mbf{y}+ L_1)$ does not change, see Figure~\ref{fig:S2a}.
	    \begin{figure}[h]
	        \centering
	        \includegraphics[width=.9\textwidth]{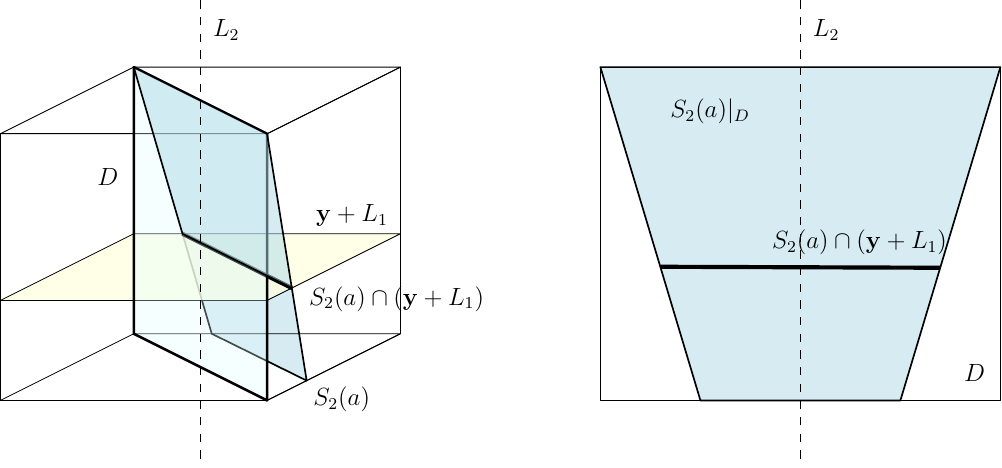}
	        \caption{Sections of $S_2(a)$ and $\mathbf{y}+L_1$ in $Q_{n-1}$ and its diagonal section $D$.}
	        \label{fig:S2a}
	    \end{figure}
	    
	    \noindent
	 Note that since by \eqref{eq:abcons}, $\frac b a \leq \frac{1}{n-2}$, \eqref{eq:S2-x} yields that for $n \geq 3$, $S_2(a)\cap(\mbf{y}+L_1)$ is non-empty for arbitrary  $\mbf{y} \in L_2 \cap Q_{n-1}$. Moreover, the length of $S_2(a)\cap(\mbf{y}+L_1)$ may be calculated from equation \eqref{eq:S2-x} using elementary plane geometry, resulting in 
     \[\vol{1}{S_2(a)\cap(\mbf{y}+L_1)}=\sqrt2-\sqrt2\cdot\frac{b}{a}\abs{\frac12-\sca{\mbf{y}}{\one{n-3}}}.\]

	    Substituting this back to \eqref{eq:vol-Proj-F2-int}, and considering the fact that ${L_2\cap Q_{n-1}=Q_{n-3}}$, we obtain that
	    \begin{equation}\label{eq:volds}
     \vol{n-2}{\Proj{D}{S_2(a)}}=\int_{Q_{n-3}}\Big(\sqrt2-\sqrt2\cdot\frac{b}{a}\cdot\abs{\frac12-\sca{\mbf{y}}{\one{n-3}}}\Big)\dd{\mbf{y}}.
        \end{equation}
     	    
            According to formula~\eqref{eq:proj}, 
	    \begin{equation*}\label{eq:Proj-F2}
	        \vol{n-2}{\Proj{D}{S_2(a)}}=\abs{\sca{\frac{\mbf{u}_2(a)}{\abs{\mbf{u}_2(a)}}}{\mbf{d}_{n-1,2}}}\cdot\vol{n-2}{S_2(a)}=\frac{\sqrt2a}{\sqrt{1-b^2}}\cdot\vol{n-2}{S_2(a)},
	    \end{equation*}
therefore, by \eqref{eq:volds},
	    \begin{equation}\label{eq:su2a}
	        s_2(a)=\frac{\sqrt{1-b^2}}{a}\cdot\bigg(1-\frac{b}{a}\cdot\int_{Q_{n-3}}\abs{\frac12-\sca{\mbf{y}}{\one{n-3}}}\dd{\mbf{y}}\bigg).
	    \end{equation}
	    Thus, by substituting \eqref{eq:su1a} and \eqref{eq:su2a} back to \eqref{eq:F=s-s}, we derive that
\begin{align*}
 F_{n,2}(a)&=\frac{1}{2a(1-a^2)}-\frac{1}{a(1-b^2)}\cdot\bigg(1-\frac{b}{a}\cdot\int_{Q_{n-3}}\abs{\frac12-\sca{\mbf{y}}{\one{n-3}}}\dd{\mbf{y}}\bigg) \geq  \\
 &\geq \frac{1}{2a(1-a^2)}-\frac{1}{a(1-b^2)}\cdot\bigg(1-\frac{b}{a}\cdot\int_{Q_{n-3}}\Big(\frac12-\sca{\mbf{y}}{\one{n-3}} \Big) \dd{\mbf{y} } \bigg) =  \\
 &=\frac{1}{2a(1-a^2)}-\frac{1}{a(1-b^2)}\cdot\Big(1-\frac{b}{a} \cdot \frac 1 2 \Big)=\\
 &= \frac{\sqrt{1 - 2 a^2}}{2 a^2 (1 - a^2) (n - 3 + 2 a^2)}\cdot \big(\sqrt{n - 2} (1 - a^2) - (n-1) a \sqrt{1 - 2 a^2}\big)
\end{align*}
by \eqref{eq:b_n2a}. Since $n \geq 4$ and $a \leq \frac{1}{\sqrt2} $, this is guaranteed to be non-negative if 
\[
(n - 2) \big(1 - a^2\big)^2 \geq (n-1)^2 a^2 \big(1 - 2 a^2\big)
\]
which holds when $a^2 \leq \frac 1 n$ or $a^2 \geq \frac{n-2}{2n-3} = \gamma_n^2$. The latter condition implies that $F_{n,2}(a) \geq 0$ for every $a \in  \big[ {\gamma_n}, \frac{1}{\sqrt2} \big]$. 
\end{proof}

\section{On the question of local extremality}\label{sec:saddle}

Supporting Conjecture~\ref{conj:all-diag}, in the concluding section we demonstrate that for any $n \geq 4$, the critical direction constructed for the proof  of Theorem~\ref{thm:nondiag} is not locally extremal on $S^{n-1}$ with respect to the central section function $\sigma(\mbf{v})$ on $S^{d-1}$. Magically, this will also prove to be a consequence of Corollary~\ref{cor:ineq}.

More precisely, let $\xi_n$ be the location of the first (and, as conjectured in Section~\ref{sec:exist-non-diag}, only) local minimum of $\widehat{\sigma}(a)$ defined by \eqref{eq:sigma-hat} on the interval $\big( \frac 1 {\sqrt{n}}, \frac 1 {\sqrt{2}}\big)$. According to the argument at the beginning of  Section~\ref{sec:exist-non-diag}, $\mbf w = \mbf v_{n,2}(\xi_n)$ is a non-diagonal critical direction, in fact, this is a zero  of $F_{n,2}(.)$ defined by \eqref{eq:F_nk}.

We will prove that $\mbf w$ is not locally extremal. By the choice of $\mbf w$, we only have to exclude the possibility of $\mbf w$ being a local minimum. According to \cite[Proposition 3.2.1]{Bertsekas}, it suffices to demonstrate that the Hessian $\widetilde H$ of the Lagrange function \eqref{eq:lagrangefunction} at $\mbf w$ is not positive definite. To that end, let $\mbf q=(1,-1,0,\ldots,0)$ and consider the value $\mbf q \widetilde H\mbf q ^T$. 
The proof of its negativity, by \eqref{eq:lagrangempl}, \eqref{eq:borderedHessian}, \eqref{eq:betavdef} and \eqref{eq:gammavdef}, amounts to showing that
 \begin{equation}\label{eq:saddle-Hesse}
            \mbf q\cdot
            \begin{bNiceMatrix}
            \beta_1(\mbf w)&\gamma_{1,2}(\mbf w)&\Cdots&\gamma_{1,n}(\mbf w)\\
            \gamma_{2,1}(\mbf w)&\beta_2(\mbf w)&\Cdots&\gamma_{2,n}(\mbf w)\\
            \Vdots&\Vdots&\Ddots&\Vdots\\
            \gamma_{n,1}(\mbf w)&\gamma_{n,2}(\mbf w)&\Cdots&\beta_{n}(\mbf w)
            \end{bNiceMatrix}
            \cdot\mbf q^T<0,
        \end{equation}
    equivalently, since $\beta_1(\mbf w)=\beta_2(\mbf w)$ and $\gamma_{1,2}(\mbf w)=\gamma_{2,1}(\mbf w)$,
    \begin{equation}\label{eq:beta-gamma}
        \beta_1(\mbf w)+\beta_2(\mbf w)-\gamma_{1,2}(\mbf w)-\gamma_{2,1}(\mbf w)=2\big(\beta_1(\mbf w)-\gamma_{1,2}(\mbf w)\big)<0.
    \end{equation}
    Based on formulae \eqref{eq:betavdef} and \eqref{eq:gammavdef},
    \begin{gather*}
        \beta_1(\mbf w)=\frac1\pi\int_{-\infty}^\infty\sinc^{n-2}(\eta_n t)\cdot\sinc(\xi_nt)\cdot\Big(\frac2{\xi_n^2}(\sinc(\xi_nt)-\cos(\xi_nt))-t^2\sinc(\xi_nt)+\sinc(\xi_nt)\Big)\dd t\\
        \gamma_{1,2}(\mbf w)=\frac1\pi\int_{-\infty}^\infty\sinc^{n-2}(\eta_nt)\cdot\Big(\frac{\cos(\xi_nt)-\sinc(\xi_nt)}{\xi_n}\Big)^2\dd t,
    \end{gather*}
    where 
    \[\eta_n=\sqrt{\frac{1-2\xi_n^2}{n-2}}.\]
    Hence on the left hand side of \eqref{eq:beta-gamma}, 
    \begin{align}\label{eq:int-xi}    
    \begin{split}
    &\beta_1(\mbf w)-\gamma_{1,2}(\mbf w)=\\
    &=\frac{1}{\pi \xi_n^2}\int_{-\infty}^\infty\sinc^{n-2}(\eta_n t)\cdot\Big((1+\xi_n^2)\sinc^2(\xi_nt)-1\Big)\dd t = \\
    &=\frac{1}{\xi_n^2}\Big((1+\xi_n^2)\sigma(\mbf w) - \sqrt{\frac{1}{{1-2\xi_n^2}}} \sigma(\mbf{d}_{n-2}) \Big).
    \end{split}
    \end{align}
By the choice of $\xi_n$, Theorem~\ref{thm:localmax} implies that $\sigma(\mbf w) < \sigma(\mbf d_n)$. Hence \eqref{eq:int-xi} is bounded from above by 
\begin{align*}
&\frac{1}{\xi_n^2}\Big((1+\xi_n^2)\sigma(\mbf w) - \sqrt{\frac{1}{{1-2\xi_n^2}}} \sigma(\mbf{d}_{n-2}) \Big) < \\
&<\frac{1}{\xi_n^2}\Big((1+\xi_n^2)\sigma(\mbf{d}_{n}) - \sqrt{\frac{1}{{1-2\xi_n^2}}} \sigma(\mbf{d}_{n-2}) \Big)  = \\
&=\frac{1}{\xi_n^2}\Big((1+\xi_n^2)\sqrt n J_n(0)-\sqrt{\frac{n-2}{{1-2\xi_n^2}}}J_{n-2}(0)\Big),
\end{align*}
where \eqref{eq:sigmaJn} is used in the last step. Due to Corollary~\ref{cor:ineq}, this is bounded from above by 
    \begin{align}\label{eq:saddle-upper}
    \begin{split}
        &\frac{1}{\xi_n^2}\bigg((1+\xi_n^2)\sqrt n\cdot\frac{n}{n+1}-\sqrt{\frac{n-2}{1-2\xi_n^2}}\bigg)J_{n-2}(0)=\\
&=         \frac{1}{\xi_n^2} \cdot \frac 1 {\sqrt{{1-2\xi_n^2}}} \cdot \frac{n^{3/2}}{n+1} \cdot J_{n-2}(0) \cdot  \bigg( \sqrt{1-2\xi_n^2} \cdot (1+\xi_n^2) - \frac{\sqrt{n-2}\cdot (n+1)}{n^{3/2}}  \bigg).
    \end{split}
    \end{align}
    Note that since $\xi_n < \frac 1 {\sqrt{2}}$, the first terms are positive, while the function $\sqrt{1-2\xi^2}\cdot(1+\xi^2)$ is strictly monotone decreasing on the interval $\xi\in\big[\frac1{\sqrt n},\frac1{\sqrt2}\big]$. Hence, \eqref{eq:saddle-upper} is maximal  at $\xi = \frac{1}{\sqrt n}$, where its value is precisely 0. As $\xi_n > \frac 1 {\sqrt{n}}$, this implies that \eqref{eq:saddle-Hesse} indeed holds.

\section{Acknowledgements}
We are very grateful to F. Fodor for his help and advice during the preparation of the article, to H. König for suspecting the existence of non-diagonal critical sections, to L. Pournin  for informing us about related works and for stimulating discussions, and to the anonymous referee whose suggestions improved the article greatly.

 \bibliographystyle{amsplain}
        \bibliography{references}

\vspace{1.5 cm}
\noindent
{\sc Gergely Ambrus}
\smallskip

\noindent
{\em Bolyai Institute, University of Szeged, Hungary, \\ and Alfréd Rényi Institute of Mathematics,  Budapest, Hungary}

\noindent
e-mail address: \texttt{ambrus@renyi.hu}

\bigskip

\noindent
{\sc Barnabás Gárgyán}
\smallskip

\noindent
{\em Bolyai Institute, University of Szeged, Hungary}

\noindent
e-mail address: \texttt{gargyan.barnabas@gmail.com}

\end{document}